\documentclass[11pt,reqno]{amsart}
\usepackage{amsfonts,amsmath,amssymb}
\newtheorem{thm}{Theorem}[section]

\newtheorem{corollary}[thm]{Corollary}
\newtheorem{lemma}[thm]{Lemma}

\newtheorem{remark}[thm]{Remark}

\usepackage{color}
\usepackage{graphicx}
\usepackage{epstopdf}

\title[Orbits and invariants]{Orbits and invariants of super Weyl groupoid}
\author{ A.N. Sergeev}\address{Department of Mathematics, Saratov State University
 Astrakhanskaya 83, Saratov 410012, Russia 
and National Research University Higher School of Economics, 
20 Myasnitskaya Ulitsa, Moscow 101000, Russia}
\email{SergeevAN@info.sgu.ru}

\author{A.P. Veselov}
\address{Department of Mathematical Sciences,
Loughborough University, Loughborough LE11 3TU, UK  and Faculty of Mechanics and Mathematics, Moscow State University, Moscow 119899, Russia}
\email{A.P.Veselov@lboro.ac.uk}

\begin{document}

\maketitle

\begin{abstract}

We study the orbits and polynomial invariants of certain affine action of the super Weyl groupoid of Lie superalgebra $\mathfrak {gl}(n,m)$, depending on a parameter. We show that for generic values of the parameter all the orbits are finite and separated by certain explicitly given invariants. We also describe explicitly the special set of parameters, for 
which the algebra of invariants is not finitely generated and does not separate the orbits, some of which are infinite. \end{abstract}

\tableofcontents

\section{Introduction}

Let $G$ be a finite group acting linearly on a finite dimensional vector space $V$ over field of characteristic zero and $P[V]^{G}$ be the algebra of the polynomial invariants. It is well-known that this algebra is finitely generated and separates the orbits, so that for any two orbits there exists an invariant $f \in P[V]^{G}$, which takes different values on these orbits, see e.g. \cite{Vinberg}, Thm. 11.105. Classical example is a finite Coxeter group generated by reflections in a real Euclidean space $V$, in which by Chevalley theorem the corresponding algebra of invariants is freely generated \cite{Hum}.

Let us consider now a finite groupoid $\mathfrak G$ acting on an affine space $V$ by partially defined affine maps (see the precise definitions in the next section). One can ask the same questions:

{\it Q1. Is the algebra of invariants $P[V]^{\mathfrak G}$ finitely generated?

Q2. Does this algebra separate the orbits of $\mathfrak G$?}

These questions seem to be closely related to the following question, which is trivial in the group case:

{\it Q3. Are all the orbits finite?}

As we will see, in general the answers to all these questions are negative.
We will demonstrate this in the case of the so called {\it super Weyl groupoid} $\mathcal{W}_{n,m}$, introduced in \cite {SV2} in relation with the Grothendieck ring of the Lie superalgebra $\mathfrak{gl} (n,m).$ We will consider a special affine action $\Phi_\kappa$ of this groupoid depending on a non-zero parameter $\kappa,$ which arosecannot from the theory of the deformed quantum Calogero-Moser systems \cite{SV1}. 

The algebra of invariants $P[V]^{\Phi_\kappa}$ for non-rational $\kappa$ is known to be finitely generated and is isomorphic to the algebra of the corresponding quantum Calogero-Moser integrals \cite{SV1}. 
A special case $\kappa=-1$ corresponds to the Lie superalgebra $\mathfrak{gl} (n,m)$, when the corresponding invariants are known as supersymmetric polynomials and generated by the characters of the polynomial representations \cite{Serge0}. The corresponding algebra is known to be not finitely generated and does not separate the orbits, some of which turned out to be infinite.

In this paper we study the situation for general $\kappa$ in more detail, in particular, for rational $\kappa.$

Let $\Phi_\kappa$ be the action of the super Weyl groupoid $\mathcal{W}_{n,m}$ described in the next section.
We say that the parameter $\kappa$ is {\it special} if \footnote{It is interesting that the same set of parameters appeared in a recent paper \cite{BCES} in connection with Cohen-Macaulay property, see more in the Concluding remarks.}
\begin{equation}
\label{specialp}
\kappa= \pm\frac{p}{q}\,\,\,\,\,  \text{for some} \,\, 1\leq p \leq m, \, 1\leq q \leq n.
\end{equation}

Our main results can be summarised as follows.

\begin{thm} 
\label{mainthm} All the orbits of $\mathcal{W}_{n,m}$ are finite if and only if $\kappa$ is not a negative special. 
In that case the algebra of polynomial invariants $P[V]^{\Phi_\kappa}$ is finitely generated. If additionally $\kappa$ is not a positive special, then this algebra separates the orbits.
\end{thm} 

For a special, explicitly described subalgebra $\Lambda_{n,m,\kappa} \subseteq P[V]^{\Phi_\kappa}$, which coincides with $P[V]^{\Phi_\kappa}$ for $\kappa$ not positive rational, we can be a bit more specific.

\begin{thm} 
The algebra $\Lambda_{n,m,\kappa}$ separates the orbits of $\Phi_\kappa$ if and only if 
$\kappa$ is not special.
\end{thm}

We conjecture that the full algebra of invariants $P[V]^{\Phi_\kappa}$ separates the orbits for positive special $\kappa$ as well. More generally, we conjecture that for any affine action of a finite groupoid with finite orbits, the invariants separate the orbits. We finish by considering some particular cases in more detail.


\section{Super Weyl groupoid and its actions}

A {\it groupoid} $\mathfrak G$ is a (small) category with all morphisms being invertible (see \cite{B,W} for the details). The set of objects is denoted as $\mathfrak B$ and called the {\it base}, the set of morphisms is usually denoted $\mathfrak G$ as groupoid itself.

If the base $\mathfrak B$ consists of one element then $\mathfrak G$ has a natural group structure. More generally, for any $x \in \mathfrak B$ one can associate an {\it isotropy group} $\mathfrak G_x$ consisting of all morphisms $g\in \mathfrak G$ from $x$ into itself. For any groupoid we have a natural equivalence relation on the base $\mathfrak B$, when $x \sim y$ if there exists a morphism $g\in \mathfrak G$ from $x$ to $y.$ 

For any set $X$ one can define the following groupoid $\mathfrak S(X)$, whose base consists of all possible subsets $Y \subset X$ and the morphisms are all possible bijections between them. By the {\it action of a groupoid $\mathfrak G$ on a set} $X$ we will mean the homomorphism  of $\mathfrak G$ into $\mathfrak S(X)$ (which is a functor between the corresponding categories). If $X$ is affine space, $Y \subset X$ are the affine subspaces and morphisms are affine bijections, then we will call it {\it affine action}.

The {\it orbit} $\mathcal O_x$ of a $\mathfrak G$-action on $X$ consists of all points $y \in X$, for which there exist elements $g_1,\dots, g_k \in \mathfrak G$ such that
$
y=g_k(g_{k-1}(\dots (g_1(x))\dots)).
$
In contrast to the group case, the product $g_k g_{k-1}\dots g_1$ in general may not exist. 

In this paper we will consider (a particular case) of the super Weyl groupoid \cite{SV2} related to any basic classical Lie superalgebra $\mathfrak g.$  The roots of $\mathfrak g$ form a generalised root system $R \subset V$ in Serganova's sense \cite{Serga}, which is a certain generalisation of the root system in the presence of the isotropic roots. For isotropic roots one cannot define the reflections, which leads to a well-known problem with defining Weyl group in this case. The reflections with respect to the non-isotropic roots generate the {\it small Weyl group} $W_0,$ which describes a partial symmetry of the system.

Consider the following groupoid $\mathfrak T_{iso}$ with the base $R_{iso},$ which is the set of all the isotropic roots in $R.$ 
The set of morphisms from $\alpha \rightarrow \beta$ is non-empty if and only if $\beta = \pm \alpha$ in which case it consists of just one element. We will denote the corresponding morphism $\alpha \rightarrow -\alpha$ as $\tau_{\alpha}, \alpha \in R_{iso}.$
The group $W_0$ is acting on $\mathfrak T_{iso}$ in a natural way: $\alpha \rightarrow w(\alpha),\,
\tau_{\alpha} \rightarrow \tau_{w(\alpha)}.$ 

The {\it super Weyl groupoid} 
\begin{equation}
\label{groupoid}
\mathcal{W}(R) = W_0 \coprod  W_0 \ltimes \mathfrak T_{iso}
\end{equation}
 is defined \cite{SV2} as the disjoint union of the group $W_0$ considered as a groupoid with a single point base $[W_0]$ and the semi-direct product groupoid $W_0 \ltimes \mathfrak T_{iso}$ with the base $R_{iso}.$ The disjoint union is a well defined operation on the groupoids \cite{B}.


In \cite{SV1} we defined the admissible deformations of generalised root systems. The roots remain the same, but the bilinear form $B$ on $V$ is deformed and depends for the classical series on one parameter $\kappa$, which is assumed to be non-zero. The case $\kappa=-1$ corresponds to the original generalised root system. 

Let $X= V$ with the deformed bilinear form $(\,, )$ and define the following affine action $\Phi_\kappa$ of the Weyl groupoid
$\mathcal{W}(R)$ on it. 
The base point $[W_0]$ maps to the whole space $V$.  Let $\alpha\in R_{iso} $ then $\tau_{\alpha}$  maps  the hyperplane $\Pi_{\alpha}$ defined by the equation $(\alpha,z)=-\frac12(\alpha,\alpha)$ into the hyperplane $\Pi_{-\alpha}$ defined by the equation $(\alpha,z)=\frac12(\alpha,\alpha)$.
The elements of the group $W_0$ are acting in a natural way and the element $\tau_{\alpha}$ acts as the shift 
\begin{equation}
\label{action}
\tau_{\alpha}(z) = z + \alpha \in \Pi_{-\alpha}, \,\, z \in \Pi_{\alpha}
\end{equation}
restricted to the hyperplane $\Pi_{\alpha}.$ 
 
As it was shown in \cite{SV1} the algebra of invariants of this action $P[V]^{\Phi_\kappa}$ is closely related to the algebra of quantum integrals of the corresponding deformed Calogero-Moser systems. In fact, the notion of the super Weyl groupoid \cite{SV2} was motivated by this relation and representation theory of classical Lie superalgebras: the Grothendieck ring $K(\mathfrak g)$ of finite dimensional representations of a basic classical Lie superalgebra $\mathfrak g$ is isomorphic to the ring of trigonometric invariants $\mathbb Z[P_0]^{\Phi_\kappa},$ where $\kappa=-1$ and $P_0$ is the weight lattice of the even part of $\mathfrak g$ (see the details in \cite{SV2}).

In the case of Lie superalgebra $\mathfrak {gl}(n,m)$ the root system in the basis 
$\varepsilon_{i},\ 1\le i\le n, \, \delta_{p}=\varepsilon_{p+n},\, 1\le p\le m$ is
$$
R_{0}=\{\varepsilon_{i}-\varepsilon_{j}, \delta_{p}-\delta_{q}, \, 1\le i \ne j\le n,\, 1\le p \ne q\le m\}, 
$$ 
$$
R_{1}=\{\pm(\varepsilon_{i}- \delta_{p}), \quad 1\le i\le n, \,
1\le p\le m\} = R_{iso}.
 $$ 
 The small Weyl group $W_0=
S_{n}\times S_{m}$ acts by separately
 permuting $\varepsilon_{i},\; i=1,\dots, n$ and
$\delta_{p},\; p=1,\dots, m.$ 
The corresponding super Weyl groupoid we denote as $\mathcal W_{n,m}.$ 
 
 The deformed bilinear form is determined by  
$$
(\varepsilon_{i},\varepsilon_{i})=1,\:
(\varepsilon_{i},\varepsilon_{j})=0,\: i\ne j,\:
(\delta_{p},\delta_{q})=\kappa,\: (\delta_{p},\delta_{q})=0,\: p\ne
q,\: (\varepsilon_{i},\delta_{p})=0. 
$$

For $\alpha= \varepsilon_{i}- \delta_{p}\in R_{iso}$ the corresponding hyperplanes $\Pi_{\pm\alpha}$ have the equations 
\begin{equation}
\label{Pi}
u_{i}-\kappa v_{p}=\pm\frac{1}{2}(1+\kappa).
\end{equation} 
The corresponding element $\tau_{\alpha}\in \mathcal W_{n,m}$ acts on the hyperplane   $\Pi_{-\alpha}$ by the formula
\begin{equation}
\label{action1}
 \tau_{\alpha}(z) =  (u_{1},\dots,u_{i}+1,\dots,u_{n},v_{1},\dots,v_{p}-1,\dots,v_{m})
\end{equation}
where $z=(u_{1},\dots,u_{n},v_{1},\dots,v_{m}).$

In the simplest example $n=m=1$ we have two lines $L_\pm$ on the plane 
defined by $$u-\kappa v=\pm\frac{1}{2}(1+\kappa)$$ and the shifts:
$(u,v) \rightarrow (u+1, v-1)$
mapping the line $L_-$ to  
$L_+$ and its inverse
$(u,v) \rightarrow (u-1, v+1).$
In this case we have a very simple groupoid $\mathfrak T_2$ with the base consisting of two points 
connected by two morphisms, which are inverse to each other. 

For $\kappa \neq -1$ the orbits are single points outside of these lines and the pairs of points
$(u,v), (u+1,v-1)$ on these lines. The case $\kappa=-1$ is special: in this case two lines are the same line $L$ given by
$u+v=0$ and preserved by the shifts. The orbits are still single points outside $L$, but on the line they are infinite and consist of
the points $(u+l,-u-l),\,\, l \in \mathbb Z.$

For general $m$ and $n$ the situation is more complicated, but as we will show now is still similar.

\section{Orbits and invariants}

By definition the algebra of invariants $P[V]^{\Phi_\kappa} \subset \mathbb{C}[u_{1},\dots,u_{n},v_{1},\dots,v_{m}]$ of the groupoid action consists of the polynomials $f$, which are symmetric in
$u_{1},\dots,u_{n}$ and $v_{1},\dots,v_{m}$ separately and satisfy
the {\it quasi-invariance conditions} (in terminology of \cite{SV1})
\begin{equation}
\label{inv}
f(u+\frac{1}{2}\varepsilon_{i},v-\frac{1}{2} \delta_{p})= f(u-\frac{1}{2}\varepsilon_{i},v+\frac{1}{2} \delta_{p})
\end{equation}
on each hyperplane $u_{i}-\kappa v_{p}=0$ for $i=1,\dots, n$ and
$p=1,\dots, m$.

This algebra first appeared in \cite{SV1} as the algebra of quantum integrals of the deformed Calogero-Moser system.
For generic $\kappa$ (which is always assumed to be non-zero) this algebra can be described as follows.

\begin{thm}  
If $\kappa$ is not positive rational, then the algebra of invariants $P[V]^{\Phi_\kappa}$ 
coincides with the algebra $\Lambda_{n,m,\kappa}$ generated by the polynomials
\begin{equation}
\label{pset}
q_{l}(x,y,\kappa)=\sum_{i=1}^n
[x_i^{l+1}-(x_i-\kappa)^{l+1}]+\sum_{p=1}^m [(y_p+1)^{l+1}-y_p^{l+1}], \,\, l \in \mathbb N,
\end{equation}
where
\begin{equation}
\label{xy}
x_i=u_i+\frac{1}{2}+\kappa, \,\,\quad y_p=\kappa v_p+\frac{\kappa}{2}.
\end{equation}
\end{thm}

\begin{proof}
Let us show first that $q_l$ are invariant under the groupoid action (\ref{action1}). This action in coordinates $(x,y)$ has a form
\begin{equation}
\label{actionxy}
(x_{1},\dots,x_{n},y_{1},\dots,y_{m}) \rightarrow (x_{1},\dots,x_{i}+1,\dots,x_{n},y_{1},\dots,y_{p}-\kappa,\dots,y_{m})
\end{equation}
on the hyperplane $x_i-y_p=0$ for $i=1,\dots,n$ and
$p=1,\dots,m$ with the usual action of $S_n\times S_m$ by permutations. The invariance of $q_l$ follows now from the obvious identity
$$
x_i^{l+1}-(x_i-\kappa)^{l+1}+(y_p+1)^{l+1}-y_p^{l+1}
$$
$$=(x_i+1)^{l+1}-(x_i+1-\kappa)^{l+1}+(y_p+1-\kappa)^{l+1}-(y_p-\kappa)^{l+1}
$$
when $x_i=y_p.$

From the results of \cite{SV1}, Section 4 it follows that for $\kappa \notin \mathbb Q_{>0}$ the algebra $P[V]^{\Phi_\kappa}$ is generated by the deformed Bernoulli sums
\begin{equation}
\label{bern}
b_{l}(u,v,\kappa)=\sum_{i=1}^n
B_{l}(u_{i}+\frac{1}{2})+\kappa^{l-1}\sum_{p=1}^m B_{l}(v_{p}+\frac{1}{2}), \,\, l \in \mathbb N,
\end{equation}
where $B_l(x)$ are classical Bernoulli polynomials.
Now the claim follows from the fact that the highest order terms of polynomials $b_{l}$ and $q_l$ in coordinates $(x,y)$ are $\kappa^{-1} p_l$ and $(l+1)p_l$ respectively, where $p_l$ are the deformed power sums
$$
p_{l}(x,y,\kappa)=\kappa(x_{1}^l+\dots+x_{n}^l)+y_{1}^l+\dots+y_{m}^l, \,\, l \in \mathbb N.
$$
\end{proof}

Recall that $\kappa$ is {\it negative special}  if $\kappa = -\frac{p}{q}$ for some $1\leq p \leq m, \, 1\leq q \leq n.$

\begin{thm} 
\label{finite}
All the orbits are finite if and only if $\kappa$ is not negative special.
\end{thm}

\begin{proof}
Note first that the polynomials $q_l(x,y,\kappa)$ are invariant for all values of parameter $\kappa.$
Since $q_l(x,y,\kappa)$ are constants on the orbits, it is enough to prove that the system
\begin{equation}
\label{system0}
q_l(x,y,\kappa)=c_l, \quad l=1,\dots, n+m
\end{equation} 
for non-special $\kappa$ has only finite number of solutions for all $c_1, \dots, c_{n+m}.$

To prove this we use the following result.

\begin{thm} ({\bf Affine B\`ezout's Theorem.})
\label{Bezout}
Let 
\begin{equation}
\label{system}
P_i(z)=0, \, i=1,\dots, N
\end{equation} 
be a polynomial system in $\mathbb C^N$, 
where $P_i(z)$ is a polynomial in $z \in \mathbb C^N$ with the highest homogeneous component 
$\bar P_i(z)$ of degree $m_i.$ If the system 
\begin{equation}
\label{system2}\bar P_i(z)=0, \, i=1,\dots, N 
\end{equation}
has only zero solution in $\mathbb C^N,$
then the system (\ref{system}) has only finitely many solutions in $\mathbb C^N$ with the sum of the multiplicities equal to $m_1\dots m_N.$
\end{thm}

\begin{proof}\footnote{We are very grateful to Askold Khovanskii for this proof and helpful comments.} 
Assume that the system (\ref{system}) has infinitely many solutions, then the corresponding 
algebraic solution set must contain an algebraic curve. The closure of the curve in the projective space $\mathbb CP^N$
must intersect the infinite hyperplane. But the intersection is described the system (\ref{system2}), 
which has no solution. Contradiction means that the system (\ref{system}) has only finitely many solutions, 
whose number is given by the classical projective B\`ezout's theorem (see e.g. \cite{Shaf}).
\end{proof}

\begin{lemma} \cite{SV1}
\label{SV}
If $\kappa$ is not negative special, then the system
$$\left\{
\begin{array}{rcl}
\kappa(x_{1}+\dots+x_{n})+y_{1}+\dots+y_{m}=0\\
\kappa(x_{1}^2+\dots+x_{n}^2)+y_{1}^2+\dots+y_{m}^2=0\\
\cdots\\
\kappa(x_{1}^{n+m}+\dots+x_{n}^{n+m})+y_{1}^{n+m}+\dots+y_{m}^{n+m}=0\\
\end{array}\right.
$$
has only zero solution in $\mathbb C^{n+m}.$ The converse is also true.
\end{lemma}

\begin{proof}
To prove this suppose that the system has non-zero solution. We can assume that all $x_i, y_j$ are non-zero. 
Let us identify equal  $x_i$ and $y_j$ as $\{z_{1}, \dots, z_{r}\}, r \le n+m, $ where all $z_j$
are different. Multiplicity of $z_j$ is a pair $(p_j, q_j)$, where
$p_j$ shows how many times $z_j$ enters the sequence $\{
x_{1},x_{2},\dots,x_{n}\}$ and $q_j$ is the same for 
the sequence $\{
y_{1},\dots,y_{m}\}$. 
Consider the first $r$ of these equations 
$$
\sum_{j=1}^{r}a_{j}z_{j}^{i}=0,\quad i=1,\dots r,
$$
where $a_{j}=\kappa p_{j}+q_{j}$ as a linear system on $a_j.$ Its determinant is of
Vandermonde type and is not zero since all $z_j$ are different and
non-zero. Hence all $a_j$ must be zero, which may happen only if $\kappa$ is negative special.

Conversely, if $\kappa = -\frac{p}{q}$ for some $1\leq p \leq m, \, 1\leq q \leq n$ then there is a non-zero solution of (\ref{system0})
with $x_1=x_2=\dots=x_q=y_1=y_2=\dots=y_p=a\neq 0$ and the rest of $x$'s and $y$'s to be zero.
\end{proof}

Combining this Lemma with the affine B\`ezout's theorem we conclude that if $\kappa$ is not negative special then the system (\ref{system0}) has only finite number of solutions for all $c_k$, and thus all the orbits are finite.

To prove the converse statement one should produce an infinite orbit for every special $\kappa=-p/q.$
Without loss of generality we can assume that $p=m, q=n.$ One can easily check that the set
$\{z=(x_1,\dots, x_n, y_1\dots,y_m)\},$
\begin{equation}
\label{orb}
x_i=a+ml+(1-i)\kappa, \quad y_j= a+ml+j-1,
\end{equation} 
with $l \in \mathbb Z, \, a \in \mathbb C$, belong to the same infinite orbit, which we denote $O(a)$.

Let us show this explicitly in the case $m=2, n=3, \kappa =-2/3$.
We have the following equivalences due to (\ref{actionxy}) 
$$
z(a):= (a,a-\kappa,a-2\kappa; a,a+1) \sim (a+1,a-\kappa,a-2\kappa; a-\kappa,a+1) 
$$
$$
\sim (a+1,a+1-\kappa,a-2\kappa; a-2\kappa,a+1)\sim (a+1,a+1-\kappa,a+1-2\kappa; a-3\kappa,a+1)
$$
$$
\sim (a+2,a+1-\kappa,a+1-2\kappa; a-3\kappa,a+1-\kappa)\sim (a+2,a+2-\kappa,a+1-2\kappa; a-3\kappa,a+1-2\kappa)
$$
$$
\sim (a+2,a+2-\kappa,a+2-2\kappa; a-3\kappa,a+1-3\kappa)=(a+2,a+2-\kappa,a+2-2\kappa; a+2,a+3)
$$
since $\kappa=-2/3.$ Thus $z(a) \sim z(a+2)$ and we have an infinite equivalence set
$$
\{(a+2l,a+2l-\kappa,a+2l-2\kappa; a+2l,a+2l+1), \, l \in \mathbb Z\} \subset O(a).
$$
The orbit $O(a)$ in general is larger than the set (\ref{orb}) (see the full orbit (\ref{orb2}) in $n=2,m=1$ case), but it is always invariant under the shift by $m$:
$
O(a+m)=O(a).
$
\end{proof}


\begin{corollary} 
\label{finitecor}
 If $\kappa$ is not negative special then any orbit consists of not more than $(m+n)!$ elements.
\end{corollary}

This upper bound is optimal: for $\kappa =1$ we have the orbits consisting of $(m+n)!$ points (which are special orbits of the group $S_{m+n}$).

As another corollary of the proof we have the following result.

\begin{thm} 
\label{finitegener}
If $\kappa$ is not negative special, then the algebras $P[V]^{\Phi_\kappa}$ and  $\Lambda_{n,m,\kappa}$ are finitely generated.
\end{thm}

\begin{proof}
We use the following well-known result from commutative algebra sometimes called {\it Artin-Tate theorem} (see Proposition 7.8 from Atiyah and Macdonald \cite{AM}):

{\it Let $A \subset B \subset C$ be rings, such that $A$ is Noetherian, $C$ is finitely generated
as $A$-algebra and as a $B$-module.
Then $B$ is a finitely generated $A$-algebra. }

Let now $A$ be the polynomial algebra generated by the first $n+m$ invariants $q_1,\dots, q_{n+m}$, $B=P[V]^{\Phi_\kappa}$ is the full algebra of invariants and $C=P[V]$ is the algebra of all polynomials on $V.$ Then from Lemma \ref{SV} it follows that 
$C$ is finitely generated module over $A,$ and hence over $B.$ This implies also that $C$ is a finitely generated $A$-algebra.
By Artin-Tate theorem $B=P[V]^{\Phi_\kappa}$ is finitely generated over $A,$ and hence over field.
The same arguments work for $B=\Lambda_{n,m,\kappa}.$
\end{proof}

\begin{remark} 
The results of \cite{SV1, BCES} imply that as the generators of the corresponding algebra $P[V]^{\Phi_\kappa}$ for generic $\kappa$ one can choose $q_l (x,y, \kappa)$ given by (\ref{pset}) with $l=1,\dots, 2mn+\min(m,n).$
\end{remark}




\section{Separation of the orbits}

Now we would like to understand whether the algebra of polynomial invariants $P[V]^{\Phi_\kappa}$ separates the orbits.
We will answer this question for its subalgebra $\Lambda_{n,m,\kappa} \subseteq P[V]^{\Phi_\kappa}$ generated by the invariants (\ref{pset}). Recall that for $\kappa$ not positive rational these two algebras coincide.

Let us define the following equivalence relation $E$: we say that
$(x,y)\sim({ \tilde x},{\tilde y})$ if and only if 
\begin{equation}
\label{condi}
q_{l}(x,y,\kappa)=q_{l}({\tilde x}, {\tilde y},\kappa)
\end{equation}
for all $l \in \mathbb N.$ We can rewrite this equivalence more conveniently as follows.

Consider the function 
\begin{equation}
\label{phi}
\varphi(t,x,y)=\frac{f(t,x)}{f(t+\kappa,x)}\frac{g(t-1,y)}{g(t,y)}
\end{equation}
where $$f(t,x)=\prod_{i=1}^n(t-x_{i}),\: g(t,y)=\prod_{j=1}^m(t-y_j).$$
It is easy to see that its logarithmic derivative
$$F(t,x,y):=\frac{d}{dt} \log \varphi(t,x,y)=\sum_{l=1}^{\infty}q_l(x,y, \kappa)t^{-l-1}.
$$
 is the generating function of $q_l(x,y, \kappa).$

As a result we have

\begin{lemma}
\label{conditions}  A pair $(x, y)$ is equivalent to $({\tilde x},{\tilde y})$
if and only if 
\begin{equation}
\label{equiv}
\varphi(t,x,y)=\varphi(t,\tilde x,\tilde y).
\end{equation}
\end{lemma}

Consider now the quotient $X_{n,m}=V/S_n\times S_m$ of the space $$V=\{(x,y): x \in \mathbb C^n, \, y \in \mathbb C^m\}$$ by the small Weyl group $S_n\times S_m$, acting by permutations of $x$'s and $y$'s separately.
We can identify this space with the space of pairs of monic polynomials
$$
X_{n,m}=\{(f(t),g(t)): f, g \in \mathbb C[t], \, \deg \, p=n,\, \deg \, q=m\}
$$
by the formulae
$$
f(t)=\prod_{i=1}^n (t - x_i), \quad g(t)=\prod_{j=1}^m (t - y_j).
$$

The equivalence relation $E$ induces the equivalence relation $(f,g)\sim({\tilde f},{\tilde g})$ on $X_{n,m}$ defined by \begin{equation}
\label{equipol}
\frac{f(t)}{f(t+\kappa)}\frac{g(t-1)}{g(t)}
=\frac{\tilde f(t)}{\tilde f(t+\kappa)}\frac{\tilde g(t-1)}{\tilde g(t)}.
\end{equation}

The action of the groupoid $\Phi_\kappa$ can be reduced to the following multivalued action of the groupoid $\mathfrak T_2$ with 
two objects and the morphisms $\tau^{\pm1}$ connecting them.
The element $\tau$ acts on a pair of polynomials $(f,g) \in X_{n,m}$ with a common root $x_i=y_p$  by
\begin{equation}
\label{equipol1}
\tau: x_i \rightarrow x_i+1, \quad y_p \rightarrow y_p-\kappa,
\end{equation} 
leaving other roots the same (see (\ref{actionxy})).
Since the polynomials may have several pairs of common roots this action is in general multivalued.
Respectively, the element $\tau^{-1}$ acts on a pair of polynomials $(f,g)$ such that $f(t)$ and $g(t-\kappa-1)$ have common root
$x_i=y_p$ by the formula
\begin{equation}
\label{equipol2}
\tau^{-1}: x_i \rightarrow x_i-1, \quad y_p \rightarrow y_p+\kappa
\end{equation} 
(and identically on the other roots). A non-symmetry between the actions of $\tau$ and $\tau^{-1}$ is due to the choice of the coordinates $(x,y).$

Let us call a pair of monic polynomials $(f(t),g(t))$ {\it minimal} if the polynomials $f(t)$ and $g(t-\kappa-1)$ have no common roots. This means that we cannot apply to such a pair the action of the "lowering" element $\tau^{-1}$.


Recall that $\kappa$ is special if $\kappa = \pm \frac{p}{q}$ for some $1\leq p \leq m, \, 1\leq q \leq n.$
 
\begin{thm} 
The algebra $\Lambda_{n,m,\kappa}$ separates the orbits of $\Phi_\kappa$ if and only if 
$\kappa$ is not special.
\end{thm}

\begin{proof} 

In \cite{SV1} it was proved that for negative special $\kappa=-\frac{p}{q}$ the algebra $\Lambda_{n,m,\kappa}=P[V]^{\Phi_\kappa}$ is not finitely generated (see Th. 5 in \cite{SV1}). 

Let us show that the orbits are not separated in this case.
Again it is enough to consider the case $\kappa=-\frac{m}{n}.$ 

Consider the infinite orbits $O(a)$ from the proof of Theorem 3.2, containing the set (\ref{orb}). Since $O(a+m)=O(a)$ is invariant under the shift by $m$, it is clear that any polynomial invariant has the value on $O(a)$, which is independent on $a.$
On the other hand, it is also obvious that $O(0)$ and, say, $O(\sqrt 2)$ are different orbits, so the orbits are indeed not separated in this case.

Let us assume now that $\kappa$ is not negative special, so all orbits are finite.


\begin{lemma}
\label{minimal}
If $\kappa$  is not negative special  then any groupoid orbit $\mathcal O$ contains a minimal pair 
$(f(t),g(t))$. 
\end{lemma}

\begin{proof} 
Consider the orbit $\mathcal O$ of the pair $(f,g)$. Assume that the pair is not minimal, and let $x$ be a root of $f(t)$ such that $g(x-\kappa-1)=0$. Consider
$$
f_{1}(t)=\frac{t-(x-1)}{t-x}f(t),\quad
g_{1}(t)=\frac{t-(x-1)}{t-(x-\kappa-1)}g(t).
$$
It is easy to see that $(f,g)\sim(f_{1},g_{1})$ and moreover $(f_{1},g_{1}) \in \tau^{-1}(f,g)$. If  $(f_{1},g_{1})$ is not minimal we can apply this procedure again etc. Since $\kappa$ is not negative special, according to Theorem \ref{finite}
the orbit $\mathcal O$  contains only finite number of elements, which means that this procedure must stop. 
The last pair of the sequence is minimal. 
\end{proof}


\begin{lemma}
\label{minimal2}
If $\kappa$ is not special then any $E$-equivalence class $\mathcal E$ contains a unique minimal pair 
$(f(t),g(t))$. 
\end{lemma}

\begin{proof} 
We use the induction in $mn.$ It is convenient to allow the cases $m=0$ and $n=0$ as well. 
In these cases the claim is obviously true since any equivalence class consists of only one element. Indeed, if for example $n=0$ then $g \equiv 1$ and from (\ref{equipol}) we have 
$$\frac{\tilde f(t+\kappa)}{f(t+\kappa)}=\frac{\tilde f(t)}{f(t)},$$ which, since $\kappa\neq 0$, implies  that $\frac{\tilde f(t)}{f(t)}\equiv 1,$ so that $\tilde f = f.$

Consider now two minimal pairs $(f,g)$ and $({\tilde f},{\tilde g})$ satisfying the relation (\ref{equipol}).
Then, using induction assumption, we can assume that $f(t)$ and $\tilde f(t)$, as well as  $g(t)$ and $\tilde g(t)$, have no common roots. 

By minimality assumption we have also that $f(t+\kappa)$ and $g(t-1)$ have no common roots.
This implies that the relation (\ref{equipol}) is equivalent to two relations
\begin{equation}
\label{equipol3}
\frac{f(t)}{f(t+\kappa)}
=\frac{\tilde g(t-1)}{\tilde g(t)}, \quad
\frac{g(t-1)}{g(t)}
=\frac{\tilde f(t)}{\tilde f(t+\kappa)}.
\end{equation}
In particular, we have 
$$
f(t) \tilde g(t)= f(t+\kappa)\tilde g(t-1).
$$
Equating the sums of the roots of both sides we have
$$
x_1+\dots+x_n+\tilde y_1+\dots+\tilde y_m= x_1+\dots+x_n+\tilde y_1+\dots+\tilde y_m -n \kappa+ m,
$$
which implies that 
$
\kappa=\frac{m}{n}
$
is special, which is a contradiction.
\end{proof}

Consider now a groupoid orbit $\mathcal O$ and the corresponding $E$-equivalence class $\mathcal E$ containing $\mathcal O.$
In general, $\mathcal E$ may consist of several orbits. However, as we have seen, for non-special $\kappa$ every such orbit
must contain a minimal pair, and such a pair is unique in $\mathcal E$. This implies that $\mathcal O=\mathcal E,$ so the algebra $\Lambda_{n,m,\kappa}$ is indeed separating the orbits.

To complete the proof of the theorem we need to show that for special $\kappa$ this is not the case. It is enough to consider the case
$\kappa=\frac{m}{n}.$ 

Let 
$$
f(t)=\prod_{i=1}^n(t-x-(i-1)\kappa), \quad g(t)=\prod_{j=1}^m (t-y-(i-1)),
$$
$$
\tilde f(t)=\prod_{i=1}^n(t-y-m+(i-1)\kappa), \quad \tilde g(t)=\prod_{j=1}^m (t-x+j).
$$
Then it is easy to see that
$$
\frac{f(t)}{f(t+\kappa)}\frac{g(t-1)}{g(t)}=\frac{t-x}{t-x+n\kappa}\frac{t-y-m}{t-y}
$$
$$
\frac{\tilde f(t)}{\tilde f(t+\kappa)}\frac{\tilde g(t-1)}{\tilde g(t)}=\frac{t-y-m}{t-y-m+n\kappa}\,\,\,\frac{t-x}{t-x+m}.
$$
But for $\kappa=\frac{m}{n}$ we have $n\kappa=m$, so $(f,g) \sim (\tilde f, \tilde g)$.
Since for generic $x$ and $y$ both these pairs are minimal, this implies that the equivalence class of these two (equivalent) pairs contains at least two different orbits, and thus the algebra $\Lambda_{n,m,\kappa}$ is not separating.
\end{proof}

As an example, let us consider in more detail the simplest case
$n=m=1.$

Assume first that $\kappa \neq \pm 1$ is non-special. 
In that case the groupoid orbits are either single point orbits, or $2$-point orbits in $x,y$ coordinates having the form
$$(x,x)\sim (x+1, x-\kappa).$$

One can show that the algebra of invariants $P[V]^{\Phi_\kappa}$ coincides in this case with $\Lambda_{1,1,\kappa}$  and is generated by
$$I_1=\kappa x + y, \,\, I_2= (x-y)(x-y-\kappa-1), \,\, I_3=x(x-y)(x-y-\kappa-1).$$ 
In particular, both algebras separate the orbits in this case.

When $\kappa=-1$ then the groupoid orbits are either single point orbits, or the infinite sets of the form
$$\mathcal O=\{(x+l, x+l), \,\, l \in \mathbb Z \}.$$
In this case the algebra of invariants 
$$P[V]^{\Phi_{\kappa}}=\{f\in \mathbb C[x,y]: f(x,y)=c+(x-y)g(x,y), \,\, c \in \mathbb C, g(x,y)\in \mathbb C[x,y]\}$$
coincides with $\Lambda_{1,1,-1}$, which is not finitely generated and clearly does not separate the infinite orbits. 

Finally, when $\kappa=1$ then the action identifies $(x,x)\sim (x+1, x-1)$, or, in the coordinates
$$u=x-1/2, \,\, v =y+1/2$$
the points $(u, u+1)\sim (u+1,u).$
The algebra $\Lambda_{1,1,\kappa}$ coincides in this case with the algebra of the symmetric polynomials:
$$\Lambda_{1,1,1}=\mathbb C[u,v]^{S_2}$$
and does not distinguish single-point orbits $(u,v)$ and $(v,u)$. The full algebra of invariants $P[V]^{\Phi_1}$ is larger in this case: it is generated by 
$$I_1= u + v, \,\, I_2= u^2+v^2, \,\, I_3=u((u-v)^2-1)$$ 
and separates all orbits.

\section{Special values of parameter}

 We consider three special values $\kappa=\pm 1$ and $\kappa=-1/2.$
 
 {\bf Case 1: $\kappa =1.$} 
 
 This case is special for all $n$ and $m.$
 One can easily see that in the original coordinates $(u,v)$ the algebra 
 $\Lambda_{n,m,1}$ coincides with the algebra of symmetric polynomials
 $\mathcal C[u,v]^{S_{n+m}}$ and thus does not distinguish the orbits 
 outside the hyperplanes $u_i-v_j=\pm 1.$ 
 The full invariant algebra $P[V]^{\Phi_1}$ contains also the whole ideal $I_\Delta$ generated by
 $$\Delta=\prod_{i=1}^n\prod_{j=1}^m((u_i-v_j)^2-1),$$
 which is enough to separate all the orbits.
 
This algebra is a subalgebra of a larger algebra of {\it quasi-invariants} $Q_{n+m}$ studied in \cite{CV, CFV} in relation with the algebra of quantum integrals of Calogero-Moser system for special value of the parameter.
 
 {\bf Case 2: $\kappa =-1.$} 
In this case the relation (\ref{equipol}) 
$$\frac{f(t)}{f(t-1)}\frac{g(t-1)}{g(t)}=\frac{\tilde f(t)}{\tilde f(t-1)}\frac{\tilde g(t-1)}{\tilde g(t)}$$
 can be rewritten as
 $$\frac{f(t)}{\tilde f(t)}\frac{\tilde g(t)}{g(t)}= \frac{f(t-1)}{\tilde f(t-1)}\frac{\tilde g(t-1)}{g(t-1)}.
 $$ 
 This implies that in that case $(f(t), g(t)) \sim (\tilde f(t), \tilde g(t))$ iff
 $$
\frac{f(t)}{g(t)}=\frac{\tilde f(t)}{\tilde g(t)}.
 $$
The corresponding algebra of invariants $P[V]^{\Phi_1}$ is known as the algebra of {\it
supersymmetric polynomials}. It plays an important role in
geometry \cite{FP} and in the representation theory of Lie superalgebra $\mathfrak {gl}(n,m)$, 
see \cite{Serge0, SV2}. It coincides with $\Lambda_{n,m,-1}$ and is generated by the (super) power sums
$$
p_l(x)=x_{1}^l+\dots+x_{n}^l-x_{n+1}^l-\dots-x_{n+m}^l, \,\, l \in \mathbb N.
$$
It is not finitely generated, and, as we have seen already in the case $n=m=1,$ does not separate the infinite orbits.

\medskip

{\bf Case 3: $\kappa =-1/2.$} 
 Let us introduce  in this case
$$
\phi(t)= \frac{f(t)}{\tilde f(t)}\frac{\tilde g(t)}{g(t)}
\frac{\tilde g(t-1/2)}{g(t-1/2)}.
$$
One can check that the equivalence relation (\ref{equipol}) means that
 $\phi(t)=\phi(t-1)$, which implies that $\phi(t)\equiv1.$
Thus in this case $(f(t), g(t)) \sim (\tilde f(t), \tilde g(t))$ iff
$$
\frac{f(t)}{g(t)g(t-1/2)}=\frac{\tilde f(t)}{\tilde g(t)\tilde
g(t-1/2)}.
 $$
The corresponding algebra of invariants and the orbits have been studied in \cite{SV5} 
in relation with the symmetric superspace $X=GL(n,2m)/OSP(n,2m).$

For $n >1$ the algebra of invariants $\Lambda_{n,m,-1/2}$ is not finitely generated 
and does not separate the orbits. 
In particular, in the case $n=2, m=1$ we have 
the union of four lines given by
$$
u_1-u_2=2u_1+v_1=\pm\frac{1}{2},\,\,\, u_2-u_1=2u_2+v_1=\pm\frac{1}{2},
$$
consisting of infinite orbits
$$
O(a)=\{(l-\frac{1}{2}+a, l+a, -2l+\frac{1}{2}-2a), (l+\frac{1}{2}+a, l+a, -2l-\frac{1}{2}-2a),
$$
\begin{equation}
\label{orb2}
(l+a, l-\frac{1}{2}+a,-2l+\frac{1}{2}-2a),
(l+a, l+\frac{1}{2}+a, -2l-\frac{1}{2}-2a)\},
\end{equation}
where $l \in \mathbb Z$ and $a$ is a parameter. It is clear that all the polynomial invariants must be constant on these lines, 
and thus cannot distinguish the orbits with different values of $a.$

The case $n=1$ is special. For example, for $n=1, m=1$ we have the one-element orbits $(u=u_1,v=v_1)$ with $u+\frac{1}{2}v \neq \pm\frac{1}{4}$ 
and two-element orbits $$(a,-\frac{1}{2}-2a) \sim (a+1, -\frac{3}{2}-2a)$$  with arbitrary parameter $a.$ In that case the algebra of invariants is finitely generated by $$f_1=u+v, f_2=(2u+v)^2, f_3=u[(2u+v)^2-1/4].$$ It contains an ideal generated by $h=(2u+v)^2-1/4$, which together with $f_1$ separates all the orbits.

\section{Concluding remarks}

Although the results of  the classical invariant theory for finite groups in general cannot be extended to finite groupoids, we have shown that in our particular case for non-special $\kappa$ the situation is similar: the algebra of invariants is finitely generated and separates the orbits. 
Moreover, we have shown that already the subalgebra $\Lambda_{n,m,\kappa}$ separates the orbits ({\it separating subalgebra} in the terminology of invariant theory, see e.g. \cite{Kemper}).

In the case of special positive $\kappa$ we still do not know if the full algebra of invariants $P[V]^{\Phi_\kappa}$ separates the orbits. We conjecture that this is true. This should follow from the following more general result.

\smallskip

{\bf Conjecture.} {\it Suppose that all the orbits of an affine action $\Phi$ of a finite groupoid $\mathfrak G$ are finite.
Then the algebra of polynomial invariants $P[V]^{\Phi}$ is finitely generated and separates the orbits.}

\smallskip

This agrees with Kollar's results \cite{Kollar} (see, in particular, Proposition 25), where more general quotients by finite equivalence relations are studied.
\footnote{We are very grateful to Richard Thomas for pointing this out to us.}

The algebras of invariants of finite groups are known to have Cohen-Macaulay property \cite{Benson}. We conjectured in \cite{SV1} that this holds also for the associated graded algebra $gr (P[V]^{\Phi_\kappa})=gr (\Lambda_{n,m,\kappa})$ for generic $\kappa,$ which was recently proved in \cite{BCES}. 
It is interesting that that the conjectured set $\tilde B (n,m)$ of exceptional values of $\kappa$ from \cite{BCES} coincides with our set of special values, for which the algebra $\Lambda_{n,m,\kappa}$ separates the orbits. To see if there is a deep reason for that is an interesting question.

In particular, we can ask the following question (see the discussion of the finite group case in \cite{Kemper}). Suppose that all the orbits of an affine action $\Phi$ of a finite groupoid $\mathfrak G$ are finite.

{\bf Question.} {\it Is it true that the algebra of invariants $P[V]^{\Phi}$ is Cohen-Macaulay? If not, does it have a separating subalgebra, which is Cohen-Macaulay?}

Following \cite{SV2} one can consider also the action of the super Weyl groupoid on complex torus $V/\mathbb Z^{n+m}$ and ask the same questions for the corresponding algebra of (exponential) invariants. In particular, 
one can consider the following algebra $\Lambda_{n,m,q,t} \subset \mathbb C[z_1,\dots, z_n, w_1,\dots, w_m]$  from \cite{SV6}, consisting of polynomials $f(z_1,\dots, z_n, w_1,\dots, w_m)$, which are symmetric in $z_1,\dots, z_n$ and $w_1,\dots, w_m$ separately and satisfy the conditions
$$
f(z_1,\dots,  qz_j, \dots, z_n, w_1,\dots, w_m)=f(z_1, \dots, z_n, w_1,\dots, tw_p, \dots, w_m)
$$
whenever $z_j=w_p$ for some $j=1,\dots,n, p=1,\dots, m.$
The parameters $q,t$ are related by $t=q^\kappa.$ As it was shown in \cite{SV6} (see Theorem 5.1), the algebra $\Lambda_{n,m,q,t}$ is finitely generated if and only if
\begin{equation}
\label{macdoq}
t^iq^j \neq 1
\end{equation}
for all $1 \leq i \leq n, \, 1\leq j \leq m,$ in agreement with (\ref{specialp}). 
For generic $q,t$ it is generated by the deformed power sums
\begin{equation}
\label{defNewton}
 p_{r}(z,w, q,t)= \sum_{i=1}^n
{z_{i}^r}+\frac{1-q^r}{1-t^r}\sum_{j=1}^m {w_{j}^r}, \quad r \in \mathbb N.
\end{equation}
Cohen-Macaulay property of the algebras generated by the generalised power sums
was studied recently by Etingof and Rains \cite{ER}.


Finally, we have discussed here only the case of super Weyl groupoid $\mathcal W_{n,m}$ corresponding to the Lie superalgebra $\mathfrak gl(n,m),$ but similar results hold for the orthosymplectic Lie superalgebra $\mathfrak {osp}(n,2m)$ as well (cf. \cite{SV1}). The situation with the exceptional basic classical Lie superalgebras could bring some new interesting phenomena.
 
\section{Acknowledgements}
We are very grateful to Alexey Bolsinov,  Misha Feigin, Askold Khovanskii and Richard Thomas for useful discussions.

This work was partially supported by the Engineering and Physical Sciences Research Council (grant EP/J00488X/1). The work of ANS was done within the framework of a subsidy granted to the National Research University Higher School of Economics by the Government of the Russian Federation for the implementation of the Global Competitiveness Programme.

ANS is grateful to the Department of Mathematical Sciences of Loughborough University for the hospitality during the autumn semesters in 2012-14, and to Saratov State University for the support of his visit to Loughborough in November 2015.


\begin{thebibliography}{99}

\bibitem{AM}
M.F. Atiyah, I.G. Macdonald {\it Introduction to Commutative Algebra.} Addison-Wesley Publ. Co., Reading Massachusetts, 1969.

\bibitem{Benson}
D.J. Benson {\it Polynomial Invariants of Finite Groups.} Cambridge Univ. Press, Cambridge, 1993.

\bibitem{BCES}
A. Brookner, D. Corwin, P. Etingof, S.V. Sam
{\it On Cohen-Macaulayness of $S_n$-invariant subspace arrangements.}
arXiv:1410.5096. Int. Math. Res. Notices, 2016(7), 2104-2126.

\bibitem{B}
R. Brown {\it From groups to groupoids: a brief survey.} Bull. London Math. Soc., {\bf 19} (1987), 113-134. 

\bibitem{CV}
O.A. Chalykh, A.P. Veselov
{\it Commutative rings of partial differential operators and Lie
algebras.} Commun. Math. Phys. {\bf 126} (1990), 597--611.

\bibitem{CFV}
O.A. Chalykh, M.V. Feigin, A.P. Veselov 
{\it Multidimensional Baker--Akhiezer Functions and Huygens' Principle.}
Commun. Math. Phys. {\bf 206} (1999), 533--566.

\bibitem{ER}
P. Etingof and E. Rains (with an appendix by M. Feigin) {\it On Cohen-Macaulayness of algebras generated by power sums.}
arXiv:1507.07485. Commun. Math. Phys., 2016.


\bibitem{FP}
W. Fulton, P. Pragacz {\it Schubert Varieties and Degeneracy
Loci.} Lect. Notes in Math. {\bf 1689}, Springer-Verlag, Berlin Heidelberg New York, 1998.

\bibitem{Hum}
J. Humphreys {\it Reflection Groups and Coxeter Groups.} 
Cambridge Univ. Press, Cambridge, 1992.

\bibitem{Kemper}
G. Kemper {\it Separating invariants.} J. Symb. Comp. {\bf 44} (2009), 1212-1222.

\bibitem{Kollar}
J. Kollar {\it Quotients by finite equivalence relations.} Current Developments in Algebraic Geometry. MSRI publ. {\bf 59} (2011), 227-256.
 Edited by L. Caporaso, J. McKernan, M. Mustata, and M. Popa.
Cambridge University Press, Cambridge, 2012.





\bibitem{Serga}
V. Serganova {\it  On generalization of root systems.} Commun. in
Algebra {\bf 24(13)}, 1996, 4281--4299.


\bibitem{Serge0}
A.N. Sergeev {\it The invariant polynomial functions on Lie superalgebras.}
C.R. de l'Academie bulgare des Sciences, {\bf 35} (1982), n.5, 573-576.



\bibitem{SV1}
A.N. Sergeev, A.P. Veselov {\it Deformed quantum Calogero-Moser problems and Lie superalgebras.} Comm. Math. Phys. {\bf 245} (2004), no. 2, 249--278.


\bibitem{SV6}
A.N. Sergeev, A.P. Veselov {\it Deformed Macdonald-Ruijsenaars operators and super Macdonald polynomials.}  Comm. Math. Phys.,  {\bf 288} (2009), 653-675.

\bibitem{SV2}
A.N. Sergeev, A.P. Veselov {\it Grothendieck  rings of basic classical  Lie superalgebras.} Annals of Mathematics {\bf 173} (2011), 663-703.

\bibitem{SV5}
A.N. Sergeev, A.P. Veselov {\it Symmetric Lie superalgebras and deformed quantum Calogero-Moser problems.} arXiv:1412.8768.

\bibitem{Shaf}
I.R. Shafarevich {\it Basic algebraic geometry.} Springer-Verlag, Berlin Heidelberg, 1977.

\bibitem{Vinberg}
E.B. Vinberg {\it A Course in Algebra.} Graduate Studies in Math., Vol. 56. Amer. Math. Soc, Providence, Rhode Island, 2003.

\bibitem{W} 
A. Weinstein {\it Groupoids: unifying internal and external symmetry.} Notices Amer. Math. Soc. {\bf 43} (1996), 744--752. 



\end{thebibliography}
\end{document}